\documentclass[12pt, reqno]{amsart}

\usepackage{amsfonts,amssymb,latexsym,amsmath, amsxtra}
\usepackage{verbatim}
\usepackage{amsthm}
\usepackage{amsmath,mathtools}
\usepackage{amscd,enumitem}
\usepackage{verbatim}
\usepackage{eurosym}
\usepackage{float}
\usepackage{color}
\usepackage{xcolor}
\usepackage{dcolumn}
\usepackage[mathscr]{eucal}
\usepackage[all]{xy}
\usepackage{bbm}
\usepackage{tikz}
\usetikzlibrary{arrows}
\usepackage[OT2,T1]{fontenc}
\DeclareSymbolFont{cyrletters}{OT2}{wncyr}{m}{n}
\DeclareMathSymbol{\Sha}{\mathalpha}{cyrletters}{"58}

\theoremstyle{plain}
\newtheorem{theorem}{Theorem}[section]
\newtheorem*{theorem*}{Theorem}

\newtheorem{lemma}[theorem]{Lemma}

\newtheorem*{conjecture*}{Conjecture}
\newtheorem*{question*}{Question}

\theoremstyle{definition}

\theoremstyle{remark}

\newtheorem*{remark*}{Remark}
\newtheorem*{remarks*}{\bf Remarks}

\newtheorem*{conj*}{Conjecture}

\theoremstyle{definition}
\theoremstyle{plain}

\newtheorem{prop}[theorem]{Proposition}

\numberwithin{equation}{section}

\newcommand{\R}{\mathbb R}
\newcommand{\N}{\mathbb N}
\newcommand{\Z}{\mathbb Z}
\newcommand{\C}{\mathbb C}

\newcommand{\Q}{{\mathbb Q}}

\makeatletter
\newcommand*{\rom}[1]{\expandafter\@slowromancap\romannumeral #1@}
\makeatother

\def\({\left(}
\def\){\right)}

\DeclareMathOperator\Log{Log}
\newcommand{\CF}{\mathcal{F}}
\newcommand{\CG}{\mathcal{G}}
\def\del{ \partial}

\numberwithin{equation}{section}
\numberwithin{theorem}{section}

\def\a{\alpha}
\def\b{\beta}
\def\g{\gamma}
\def\l{\lambda}
\def\lp{\left(}
\def\rp{\right)}

\usepackage{centernot}

\makeatletter
\@namedef{subjclassname@2020}{%
	\textup{2020} Mathematics Subject Classification}
\makeatother

\title[Asymptotics for $d$-fold partition diamonds]{Asymptotics for $d$-fold partition diamonds and related infinite products}

\author{Kathrin Bringmann}
\address{Department of Mathematics and Computer Science\\Division of Mathematics\\University of Cologne\\ Weyertal 86-90 \\ 50931 Cologne \\Germany}
\email{kbringma@math.uni-koeln.de}

\author{William Craig}
\address{Department of Mathematics and Computer Science\\Division of Mathematics\\University of Cologne\\ Weyertal 86-90 \\ 50931 Cologne \\Germany}
\email{wcraig@uni-koeln.de}

\author{Joshua Males}
\address{School of Mathematics, University of Bristol, Bristol, BS8 1TW, UK, and the Heilbronn Institute for Mathematical Research, Bristol, UK.}
\email{joshua.males@bristol.ac.uk}

\begin{document}
	
	\begin{abstract}
	We prove an asymptotic formula for the number of $d$-fold partition diamonds of $n$ and their Schmidt-type counterparts. In order to do so, we study the asymptotic behavior of certain infinite products. We also remark on interesting potential connections with mathematical physics and Bloch groups.
	\end{abstract}
	
	
	
	
\maketitle
	
\section{Introduction and statement of results}

A {\it partition} of a non-negative integer $n$ is a finite sequence $\lambda = \lp a_0, a_1, \cdots, a_k \rp$ of positive integers such that $\left| \lambda \right| := a_0 + a_1 + \dots + a_k = n$. The theory of partitions has a long and rich history in combinatorics and number theory, which is overviewed in Andrews' book \cite{Andrews}. In this paper, we are primarily concerned with the asymptotic properties of partitions. The modern viewpoint on this study began with the famous paper of Hardy and Ramanujan \cite{HR}, in which they studied the function $p(n)$ which counts the number of partitions of $n$ and proved that
\begin{align*}
	p(n) \sim \dfrac{1}{4n\sqrt{3}} e^{\pi \sqrt{\frac{2n}{3}}} \quad \text{(as $n\to\infty$)}.
\end{align*}
They showed this theorem by developing the Circle Method, which has since spawned a huge number of variations with applications across all of analytic number theory.

In 2001, Andrews, Paule, and Riese \cite{APR} reinitiated the study of {partition analysis}, an algebraic framework designed by MacMahon for the deduction of generating functions for different kinds of plane partitions. This began a long series of papers from Andrews and collaborators on this topic, including papers on hypergeometric multisums \cite{AP IV}, magic squares \cite{APRS V} and recently partitions with $n$ copies of $n$ \cite{AP XIV}. In particular, this new research spawned a great interest in plane partition diamonds. A {\it plane partition diamond} (or just a {\it partition diamond}) as defined in \cite{APR VIII} is a pair of sequences of integers $\{a_j\}_{j \geq 0}$, $\{b_j\}_{j \geq 0}$ such that for every $j\in\N_0$ we have $a_j \geq \max\{b_{2j}, b_{2j+1}\}\ge a_{j+1}$. The naming convention of partition diamonds comes from the fact that these can be represented graphically as a kind of directed graph of diamonds, with the direction of edges denoting the inequalities imposed. Partition diamonds have been the subject of many very interesting studies; for example, they give examples of modular forms \cite{AP XIII, APR XI}, and their generalizations exhibit interesting congruence properties \cite{APR VIII, Chan, PR, RS}.

In this paper, we consider a recent generalization of plane partition diamonds. In \cite{DJSW}, a {\it $d$-fold partition diamond} is defined as a collection of non-negative integer sequences $\{a_k\}_{k \geq 0}, \{b_{j,k}\}_{k \geq 0, \, 0 \leq j \leq d-1}$ such that for every $k\in\N_0$, we have the inequalities $a_k \geq \max\limits_{0 \leq j \leq d-1} b_{j,k} \geq a_{k+1}$. Observe that standard integer partitions can be viewed as 1-fold partition diamonds, and the previously defined plane partition diamonds can be viewed as 2-fold partition diamonds. The inequalities exhibited by a $d$-fold partition diamond can also be represented cleanly with a directed graph; we exhibit how this works for $4$-fold diamond partitions below.

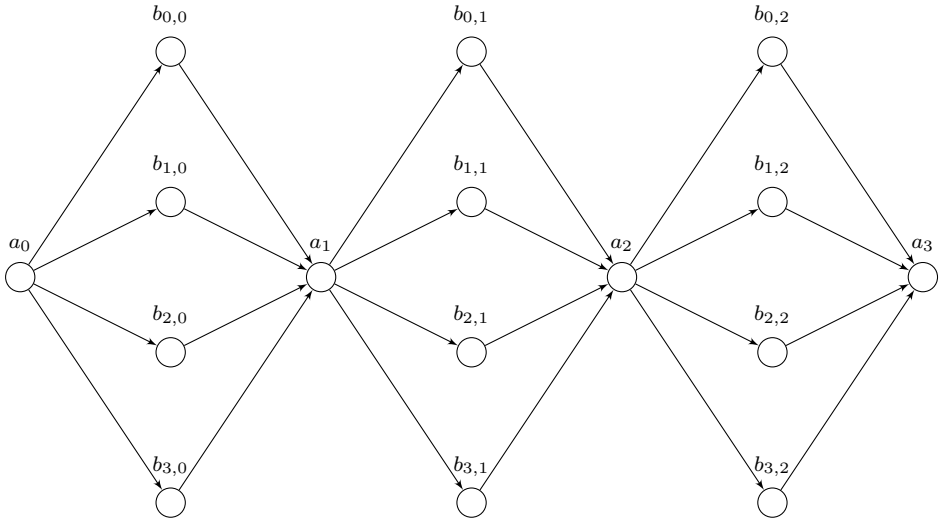
\begin{tikzpicture}
	\tikzset{vertex/.style = {shape=circle,draw,minimum size=0.5em}}
	\tikzset{edge/.style = {->,> = latex'}}

	\node[vertex] (a0) [label=above:{{\tiny $a_0$}}] at (0,0) {};
	\node[vertex] (a1) [label=above:{{\tiny $a_1$}}] at (4,0) {};
	\node[vertex] (a2) [label=above:{{\tiny $a_2$}}] at (8,0) {};
	\node[vertex] (a3) [label=above:{{\tiny $a_3$}}] at (12,0) {};
	
	\node[vertex] (b00) [label=above:{{\tiny $b_{0,0}$}}] at (2,3) {};
	\node[vertex] (b10) [label=above:{{\tiny $b_{1,0}$}}] at (2,1) {};
	\node[vertex] (b20) [label=above:{{\tiny $b_{2,0}$}}] at (2,-1) {};
	\node[vertex] (b30) [label=above:{{\tiny $b_{3,0}$}}] at (2,-3) {};
	\node[vertex] (b01) [label=above:{{\tiny $b_{0,1}$}}] at (6,3) {};
	\node[vertex] (b11) [label=above:{{\tiny $b_{1,1}$}}] at (6,1) {};
	\node[vertex] (b21) [label=above:{{\tiny $b_{2,1}$}}] at (6,-1) {};
	\node[vertex] (b31) [label=above:{{\tiny $b_{3,1}$}}] at (6,-3) {};
	\node[vertex] (b02) [label=above:{{\tiny $b_{0,2}$}}] at (10,3) {};
	\node[vertex] (b12) [label=above:{{\tiny $b_{1,2}$}}] at (10,1) {};
	\node[vertex] (b22) [label=above:{{\tiny $b_{2,2}$}}] at (10,-1) {};
	\node[vertex] (b32) [label=above:{{\tiny $b_{3,2}$}}] at (10,-3) {};
	
	\draw[edge] (a0) to (b00);
	\draw[edge] (a0) to (b10);
	\draw[edge] (a0) to (b20);
	\draw[edge] (a0) to (b30);
	\draw[edge] (a1) to (b01);
	\draw[edge] (a1) to (b11);
	\draw[edge] (a1) to (b21);
	\draw[edge] (a1) to (b31);
	\draw[edge] (a2) to (b02);
	\draw[edge] (a2) to (b12);
	\draw[edge] (a2) to (b22);
	\draw[edge] (a2) to (b32);
	
	\draw[edge] (b00) to (a1);
	\draw[edge] (b10) to (a1);
	\draw[edge] (b20) to (a1);
	\draw[edge] (b30) to (a1);
	\draw[edge] (b01) to (a2);
	\draw[edge] (b11) to (a2);
	\draw[edge] (b21) to (a2);
	\draw[edge] (b31) to (a2);
	\draw[edge] (b02) to (a3);
	\draw[edge] (b12) to (a3);
	\draw[edge] (b22) to (a3);
	\draw[edge] (b32) to (a3);
\end{tikzpicture}

In line with recent work on Schmidt-type partitions of $n$ \cite{AP XIV}, we define the {\it Schmidt size} of a $d$-fold partition diamond $\{a_k\}_{k \geq 0}, \{b_{j,k}\}_{k \geq 0,\, 0 \leq j \leq d-1}$ as the size of the subpartition $\{a_k\}_{k\ge0}$; in terms of the directed graph above, the Schmidt size of a $d$-fold partition diamond is the sum of the central nodes. Questions related to Schmidt-style modified size functions on partitions have been popular recently in the theory of partitions \cite{Alladi, dHS, Konan, LY}.

In this paper, we consider the functions that count $d$-fold partition diamonds of size $n$ and Schmidt size $n$ and compute their asymptotic expansions. In line with \cite{DJSW}, we define by $r_d(n)$ the number of $d$-fold partition diamonds of $n$ and we let $s_d(n)$ be the number of $d$-fold partition diamonds with Schmidt size $n$. In order to state these asymptotic formulas, we need to define certain constants. Let
\begin{equation*}
	C_{d}:=\int_{0}^{\infty} \log\left(A_{d}\left(e^{-x}\right)\right)dx,
\end{equation*}
with $A_d(x)$ the Eulerian polynomials defined in \eqref{Eulerian Polynomial Definition}. Then we have the following.

\begin{theorem}\label{Thm: s_d}
	As $n \to \infty$ we have that
	\begin{align*}
			s_d(n) \sim \frac{\left(C_d+\frac{\pi^2(d+1)}6\right)^{\frac d4+\frac12}}{\sqrt2(2\pi)^{\frac d2+1}\sqrt{d!} n^{\frac d4+1}} e^{2\sqrt{\left(C_d+\frac{\pi^2(d+1)}6\right)n}}.
	\end{align*}
\end{theorem}

Our second main result is the following theorem.

\begin{theorem}\label{Thm: r_d}
	As $n \to \infty$ we have that
	\begin{align*}
		r_d(n) \sim \frac{ \left(\frac{C_d}{d+1} +\frac{\pi^2}{6}\right)^{\frac{1}{2}} e^{\frac{d-1}{2(d+1)} d!} }{2\sqrt{2}\pi d!^{\frac{d}{2(d+1)}}   n } e^{2\sqrt{\left(\frac{C_d}{d+1} + \frac{\pi^2}{6}\right) n}}.
	\end{align*}
\end{theorem}

The remainder of our paper is laid out as follows. In Section \ref{Asymptotic techniques} we outline the main asymptotic techniques we apply in our analysis. In Section \ref{Preliminaries}, we explain certain preliminary facts about Eulerian polynomials and certain two-variable deformations of Eulerian polynomials given in \cite{DJSW}, and give evaluations of certain integrals which emerge in the process of proving the main theorems. In Sections \ref{Proof of Theorem 1} and \ref{Proof of Theorem 2}, we prove Theorems \ref{Thm: s_d} and \ref{Thm: r_d}, respectively, as well as very broad generalizations of these results. Finally, in Section \ref{Final Remarks} we discuss some final remarks, including possible applications to physics and connections of certain constants in our formulas with an open question about Bloch groups.

\section*{Acknowledgments}

The first and second authors have received funding from the European Research Council (ERC) under the European Union's Horizon 2020 research and innovation programme (grant agreement No. 101001179). The authors thank Steven Charlton, Caner Nazaroglu, and Don Zagier for helpful conversations.

\section{Asymptotic techniques} \label{Asymptotic techniques}

\subsection{A variation of Euler--Maclaurin summation}

We say that a function $f$ is of {\it sufficient decay} in an (unbounded) domain $D\subset\C$ if there exists some $\varepsilon > 0$ such that $f(w) \ll w^{-1-\varepsilon}$ as $|w| \to \infty$ in $D$. We need to use a version of Euler--Maclaurin summation which has been popularized by Zagier \cite{Z-EM}. We quote Theorem 1.2 of \cite{BJM} which follows from the Euler--Maclaurin summation formula.

\begin{prop}\label{Theorem:EulerMaclaurin1DShifted}
	Suppose that $0\le \theta < \frac{\pi}{2}$ and let
	$D_\theta := \{ re^{i\alpha} : r\ge0 \mbox{ and } |\alpha|\le \theta  \}$.
	Let $f:\C\rightarrow\C$ be holomorphic in a domain containing
	$D_\theta$, so that in particular $f$ is holomorphic at the origin, and
	assume that $f$ and all of its derivatives are of sufficient decay.
	Then for $a\in\mathbb{R}$ and $N\in\N_0$,
	\begin{equation*}
		\sum_{m\geq0}f((m+a)w) = \frac{I_{f}}{w} - \sum_{n=0}^{N-1} \frac{B_{n+1}(a) f^{(n)}(0)}{(n+1)!}w^n + O_N\left(w^N\right),
	\end{equation*}
	uniformly, as $w\rightarrow0$ in $D_\theta$. Here $I_{f}:=\int_0^\infty f(x) dx$.
\end{prop}
We also require the more general result, which is given in the proof of Theorem 1.2 of \cite{BJM} (see equation (5.8) there).
\begin{prop}\label{Thm: Euler-Mac full}
	Assume the conditions from Proposition \ref{Theorem:EulerMaclaurin1DShifted} are satisfied. We have for any $N \in \mathbb{N}$ that
	\begin{multline*}
		\sum_{m\geq0} f((m+a)w) = \frac{I_f}{w} - \sum_{n=0}^{N-1} \frac{B_{n+1}(a)f^{(n)}(0)}{(n+1)!} w^n - \sum_{k\ge N} \frac{f^{(k)}(0)a^{k+1}}{(k+1)!}w^k\\
		- \frac{w^N}{2\pi i} \sum_{n=0}^{N-1} \frac{B_{n+1}(0)a^{N-n}}{(n+1)!} \int_{C_R(0)} \frac{f^{(n)}(z)}{z^{N-n}(z-aw)} dz - (-w)^{N-1} \hspace{-0.1cm} \int_{aw}^{w\infty} \frac{f^{(N)}(z)\widetilde B_N\left(\frac zw-a\right)}{N!} dz,
	\end{multline*}
	where $\widetilde B_n(x):=B_n(x-\lfloor x\rfloor)$ and $C_R(0)$ denotes the circle of radius $R$ centred at the origin, where $R$ is such that $f$ is holomorphic in $C_{R}(0)$.
\end{prop}

\subsection{Ingham's Tauberian theorem}

In order to compute the asymptotic behavior of the coefficients $s_d(n)$ and $r_d(n)$ as $n \to \infty$, we make use of a Tauberian Theorem variant proved by Jennings-Shaffer, Mahlburg, and the first author, following work of Ingham. In essence, for generating functions carrying certain analytic properties\footnote{The second condition is often dropped in \eqref{E:as1} which makes the proposition unfortunately incorrect (see \cite{BJM}).} this gives an easy-to-use method of obtaining the main-term asymptotic of its Fourier coefficients. We quote the special case $\a=0$ of Theorem 1.1 of \cite{BJM}, which follows from Ingham's Theorem \cite{I}.

\begin{prop}\label{P:CorToIngham}
	Let $B(q)=\sum_{n\ge0}b(n)q^n$ be a power series with non-negative real coefficients and radius of convergence at least one and that the $b(n)$ are weakly increasing. Assume that $\l$, $\b$, $\g\in\R$ with $\g>0$ exist such that
	\begin{equation}\label{E:as1}
		B\left(e^{-t}\right) \sim \l t^\b e^\frac\g t \quad\text{as } t \to 0^+,\qquad B\left(e^{-z}\right) \ll |z|^\b e^\frac{\g}{|z|} \quad\text{as } z \to 0,
	\end{equation}
	with $z=x+iy$ ($x,y\in\R,x>0$) in each region of the form $|y|\le\Delta x$ for $\Delta>0$. Then
	\begin{align*}
	b(n) \sim \frac{\l\g^{\frac\b2+\frac14}}{2\sqrt\pi n^{\frac\b2+\frac34}}e^{2\sqrt{\g n}} \qquad\text{as } n \to \infty.
	\end{align*}
\end{prop}

To use Proposition \ref{P:CorToIngham} to study the asymptotic growth of $r_d(n)$ and $s_d(n)$, we need to verify that these are weakly increasing. We quickly prove that these properties hold.

\begin{lemma} \label{L:Weakly Increasing}
	For $d\in\N$, the sequences $s_d(n)$ and $r_d(n)$ are weakly increasing.
\end{lemma}

\begin{proof}
	Let $\mathcal R_d(n)$ and $\mathcal S_d(n)$ be the collections of $d$-fold partition diamonds of size and Schmidt-size $n$, respectively, so that $r_d(n) = \left| \mathcal R_d(n) \right|, s_d(n) = \left| \mathcal S_d(n) \right|$. It is enough to construct injections $\mathcal R_d(n) \hookrightarrow \mathcal R_d(n+1)$ and $\mathcal S_d(n) \hookrightarrow \mathcal S_d(n+1)$. Such a map is immediately furnished in both cases by that function which takes a $d$-fold diamond partition $(\{a_k\}_{k \geq 0}, \{b_{j,k}\}_{j,k})$ and adds $1$ to $a_0$ and leaves all other part sizes fixed.
\end{proof}

\section{Preliminaries} \label{Preliminaries}

\subsection{Asymptotics of the $q$-Pochhammer symbol}

We recall the famous asymptotic formula for the inverse of $q$-Pochhammer symbol, which follows from the modularity of the Dedekind $\eta$-function, and is given by
\begin{equation}\label{E:partas}
	\frac{1}{\left(e^{-z};e^{-z}\right)_\infty} \sim \sqrt\frac{z}{2\pi}e^\frac{\pi^2}{6z} \qquad\text{as } z \to 0.
\end{equation}

\subsection{Eulerian polynomials}\label{Sec: Eulerain polys}

We consider here the {\it Eulerian polynomials}, which we denote by $A_d(x)$, and some of their basic properties. For more properties and proofs, see  \cite[26.14]{DLMF}. These polynomials can be defined by the power series identity
\begin{equation} \label{Eulerian Polynomial Definition}
	\sum_{j \geq 0} \lp j+1 \rp^d x^j = \frac{A_d(x)}{\lp 1 - x \rp^{d+1}}.
\end{equation}
Based on this property, these polynomials can also be defined recursively by $A_0(x) = 1$ and for each $d\in\N$,
\begin{equation} \label{Eulerian Polynomial Recursion}
	A_d(x) = \lp 1 + (d-1)x \rp A_{d-1}(x) + x(1-x) A_{d-1}^\prime(x).
\end{equation}
The first few Eulerian polynomials are
\begin{equation*}
	A_{1}(x)=1, \quad A_{2}(x)=1+x, \quad A_{3}(x)=1+4x+x^{2}.
\end{equation*}
We require a few special values of these polynomials. In particular, by induction on $d$ it is not hard to prove that 
\begin{align}\label{eqn: A_d(1)}
	A_d(1) = d!.
\end{align}
We also obtain by differentiating \eqref{Eulerian Polynomial Recursion} and induction on $d$ that\footnote{We note that $A_d^\prime(1)$ are known as the {Lah numbers} (OEIS A001286).}
\begin{equation} \label{Eulerian Polynomial Special Values}
	A_d^\prime(1) = \frac{(d-1) \cdot d!}{2}.
\end{equation}
We need to use an important and well-known symmetry property of the Eulerian polynomials: for $d\in\N$, we have
\begin{equation}\label{Eulerian Polynomial Symmetry}
	A_d(x) = x^{d-1} A_d\left(\frac1x\right).
\end{equation}
We make use of the following lemma, which follows directly from \eqref{Eulerian Polynomial Definition}.

\begin{lemma} \label{L:A_d Zeros}
	For $d\in\N$, $A_d(x)$ has no zeros in $[0,1]$.
\end{lemma}


\subsection{Deformed Eulerian polynomials}

In order to analyze $d$-fold partition diamonds, we need to consider certain polynomials $F_d(x,y)$ which were introduced in \cite{DJSW}. These polynomials are defined recursively by
\begin{equation*} 
	F_1(x,y) = 1, \ \ \ F_d(x,y) = \frac{\lp 1 - xy^d \rp F_{d-1}(x,y) - y\lp 1-x \rp F_{d-1}(xy,y)}{1 - y}.
\end{equation*}
The first cases are
\begin{equation*}
	F_{2}(x,y)=1+xy, \quad F_{3}(x,y)=1+2xy+2xy^{2}+x^{2}y^{3}.
\end{equation*}
For later convenience, we set
\begin{equation*}
	H_d(x,y) := \lp 1 - xy^d \rp F_{d-1}(x,y) - y\lp 1-x \rp F_{d-1}(xy,y),
\end{equation*}
so that $F_d(x,y) = \frac{H_d(x,y)}{1-y}$. We refer to $F_d(x,y)$ as a deformation of the Eulerian polynomials because of the following lemma.

\begin{lemma} \label{F_d(x,1) = A_d(x)}
	We have for $d\in\N$ that $F_d(x,1) = A_d(x)$.
\end{lemma}

\begin{proof}
	For this proof and for later convenience, we observe that by simple differentiation rules that 
	\begin{multline} \label{H_d^{(0,1)}(x,y) equation}
		- H_d^{(0,1)}(x,y) = dxy^{d-1} F_{d-1}(x,y) - \left(1-xy^d\right) F_{d-1}^{(0,1)}(x,y) + (1-x)F_{d-1}(xy,y)\\
		+ xy(1-x) F_{d-1}^{(1,0)}(xy,y)
		+ y(1-x) F_{d-1}^{(0,1)}(xy,y).
	\end{multline}
	Now, $F_d(x,1) = \lim\limits_{y \to 1} \frac{H_d(x,y)}{1-y}$, and we can apply L'Hopitals rule to obtain
	\begin{align*}
		F_d(x,1) = - \lim\limits_{y \to 1} H_d^{(0,1)}(x,y) = \lp 1 + (d-1)x \rp F_{d-1}(x,1) + x\lp 1-x \rp F_{d-1}^{(1,0)}(x,1).
	\end{align*}
	Observing that this recurrence matches \eqref{Eulerian Polynomial Recursion} and that $F_1(x,1) = A_1(x) = 1$, the claim follows.
\end{proof}

We need a brief lemma which specifies that $F_d(x,y)$ does not have zeros of a certain type. This lemma follows by combining the fact that $F_d(x,y)$ is continuous with Lemma \ref{L:A_d Zeros} and Lemma \ref{F_d(x,1) = A_d(x)}.
\begin{lemma} \label{L:F_d zeros}
	For $d \in \mathbb{N}$, there exists a neighborhood $\mathcal{N}_{d}$ of $y = 1$ such that $F_d(x,y) \not = 0$ for all $x \in [0,1]$ and $y \in \mathcal{N}_{d}$.
\end{lemma}

We also need a certain differential equation satisfied by $F_d(x,y)$, which is centrally important for the evaluation of the asymptotic expansion of $r_d(n)$.

\begin{lemma} \label{F_d Differential Equation}
	We have for $d\in\N$ that
	\begin{align*}
		F_d^{(0,1)}(x,1) = \frac{dx}{2} F_d^{(1,0)}(x,1).
	\end{align*}
\end{lemma}

\begin{proof}
	We prove this claim by induction on $d$. The identity is clear for $d=1$. We next assume that for fixed $d\ge2$, the claim holds. We next reduce the claim to an expression in terms of $H_d(x,y)$. By using L'Hopital's rule, it is not hard to see that
	\begin{align*}
		F_d^{(1,0)}(x,1) &= \lim_{y \to 1} \frac{H_d^{(1,0)}(x,y)}{1-y} = - H_d^{(1,1)}(x,1),\\
		F_d^{(0,1)}(x,1) &= \lim_{y \to 1} \frac{(1-y) H_d^{(0,1)}(x,y) + H_d(x,y)}{\lp 1-y \rp^2} = 
		- \frac{1}{2} H_d^{(0,2)}(x,1).
	\end{align*}
	Therefore, in order to prove the lemma we only need to prove that
	\begin{align} \label{H-differential equation}
		- H_d^{(0,2)}(x,1) = - dx H_d^{(1,1)}(x,1).
	\end{align}
	We use \eqref{H_d^{(0,1)}(x,y) equation} as a stepping stone for proving \eqref{H-differential equation}. By taking the derivative of \eqref{H_d^{(0,1)}(x,y) equation} with respect to $x$ and evaluating subsequently at $y=1$, it is not hard to see that
	\begin{multline*}
		\hspace{-0.4cm}- dx H_d^{(1,1)}(x,1)\\ = d(d-1)x F_{d-1}(x,1) + dx\lp 2 + \lp d-3 \rp x \rp F_{d-1}^{(1,0)}(x,1) + dx^2\lp 1-x \rp F_{d-1}^{(2,0)}(x,1).
	\end{multline*}
	Similarly, by taking the derivative of \eqref{H_d^{(0,1)}(x,y) equation} with respect to $y$ and substituting $y=1$, we obtain
	\begin{multline} \label{Temp H2 calculation}
		- H_d^{(0,2)}(x,1) = d(d-1)x F_{d-1}(x,1) + 2\lp \lp d-1 \rp x + 1 \rp F_{d-1}^{(0,1)}(x,1)\\
		+ 2x\lp 1-x \rp F_{d-1}^{(1,0)}(x,1)
		+ 2x\lp 1-x \rp F_{d-1}^{(1,1)}(x,1) + x^2\lp 1-x \rp F_{d-1}^{(2,0)}(x,1).
	\end{multline}
	Now, using the induction hypothesis, we can show that
	\begin{align*}
		F_{d-1}^{(1,1)}(x,1) &= \frac{\partial}{\partial x} F_{d-1}^{(0,1)}(x,1) = \frac{\partial}{\partial x}  \frac{(d-1)x}{2} F_{d-1}^{(1,0)}(x,1)\\
		&= \frac{d-1}{2} F_{d-1}^{(1,0)}(x,1) + \frac{(d-1)x}{2} F_{d-1}^{(2,0)}(x,1).
	\end{align*}
	Substituting this into \eqref{Temp H2 calculation} and comparing the formula to that for $-dx H_d^{(1,1)}(x,1)$, we obtain
	\eqref{H-differential equation} and therefore the lemma is proven.
\end{proof}

\subsection{Generating functions for $s_d(n)$ and $r_d(n)$}

Here we recall the generating functions for $s_d(n)$ and $r_d(n)$, each of which were proven in \cite{DJSW}. Firstly, for $s_d(n)$ Theorem 1.2 of \cite{DJSW} gives that
\begin{align}\label{eqn: gen function s_d}
	\sum_{n\geq0} s_d(n) q^n = \prod_{n\ge1} \frac{A_d\left(q^n\right)}{\left(1-q^n\right)^{d+1}}.
\end{align}
We also require the generating function for $r_d(n)$, proven in Theorem 1.1 of \cite{DJSW},
\begin{align}\label{eqn: gen function r_d}
	\sum_{n\geq0} r_d(n) q^n = \prod_{n\ge1} \frac{F_d\left( q^{(d+1)(n-1)+1} , q \right)}{1-q^n}.
\end{align}

\subsection{Evaluating integrals}

In the process of evaluating the constants in our main theorems, the evaluation of certain integrals is paramount. In order  to evaluate these integrals, we need the {\it dilogarithm function} defined for $|z| \leq 1$ by
\begin{equation*} 
	\mathrm{Li}_2\lp z \rp := \sum_{n \geq 1} \frac{z^n}{n^2}
\end{equation*}
and on $\C \setminus [1,\infty)$ by the analytic continuation (see \cite[page 5]{Z})
	 \begin{align*}
	 		\mathrm{Li}_2\lp z \rp \coloneqq - \int_{0}^{z} \log(1-u)\frac{du}{u}.
	 \end{align*}
	 For many interesting properties of this function, see \cite{Z}. We now prove a proposition to evaluate certain integrals.

\begin{prop} \label{Main Term Integrals}
	Let $P(x) = \prod_{j=1}^d \lp x - \alpha_j \rp \in \R[x]$ be a monic polynomial of degree $d\in\N$ such that $P(0) = 1$ and such that $P(x)$ has no zeros on the interval $[0,1]$. Define the integrals
	\begin{align*}
		\mathcal I_P := \int_0^\infty \Log\lp P(e^{-x}) \rp dx.
	\end{align*}
	Then we have
	\begin{align*}
		\mathcal I_P = - \sum_{j=1}^d \mathrm{Li}_2\lp \frac{1}{\alpha_j} \rp.
	\end{align*}
\end{prop}

\begin{proof}
	Observe firstly that the assumptions that $P$ is monic, that $P(0) = 1$ and that $P(x)$ has no zeros in the interval $[0,1]$ imply that $\mathcal I_P$ converges. Using integration by parts, we obtain
	\begin{align*}
		\mathcal I_P = \int_0^\infty \frac{x e^{-x} P^\prime(e^{-x})}{P(e^{-x})} dx.
	\end{align*}
	By further substituting $u = e^{-x}$, we obtain
	\begin{align*}
		\mathcal I_P = - \int_0^1 \log(u) \frac{P^\prime(u)}{P(u)} du.
	\end{align*}
	Since $P$ is a monic polynomial, we have
	\begin{align*}
		\frac{P^\prime(u)}{P(u)} = \sum_{j=1}^d \frac{1}{u-\alpha_j},
	\end{align*}
	and therefore
	\begin{align*}
		\mathcal I_P = - \sum_{j=1}^d \int_0^1 \frac{\log(u)}{u-\alpha_j} du.
	\end{align*}
	We now consider for $a \not \in [0,1]$ the integrals
	\begin{align*}
		I(a) := \int_0^1 \frac{\log(u)}{u-a} du.
	\end{align*}
	We claim that $I(a) = \mathrm{Li}_2(\frac1a)$. Because $\mathrm{Li}_2(z)$ is analytic in $\C\setminus[1,\infty)$ (see \cite{Z}), both sides of this formula are analytic functions of $a$ away from $[0,1]$, and therefore to prove our claim we only need to prove its truth in the region $a>1$. Here, the identity $\frac{d}{du} \mathrm{Li}_2\lp u \rp = - \frac{1}{u} \log\lp 1-u \rp$ is valid for $|u|<1$ because of the series expansion of $\mathrm{Li}_2(u)$, and so it is straightforward to show that
	\begin{align*}
		\dfrac{d}{du} \left( \mathrm{Li}_2\lp \dfrac{u}{a} \rp + \log(u) \log\lp 1 - \dfrac{u}{a} \rp \right) = \dfrac{\log(u)}{u-a}
	\end{align*}
	for $a>1$. Since $\mathrm{Li}_2(0) = 0$ and $\log(u)\log(1-\frac ua) \to 0$ as $u \to 0^+$, we therefore obtain for $a > 1$ that
	\begin{align*}
		\int_0^1 \dfrac{\log(u)}{u-a} du = \mathrm{Li}_2\lp \dfrac{1}{a} \rp,
	\end{align*}
	and by analytic continuation the identity holds for $a \in \C \backslash [0,1]$. Since the polynomial $P$ has no zeros in the interval $[0,1]$, the claim follows.
\end{proof}

\section{Proof of Theorem \ref{Thm: s_d}} \label{Proof of Theorem 1}

In this section we prove Theorem \ref{Thm: s_d}. Recall the generating function for $s_d(n)$ in \eqref{eqn: gen function s_d}. To ease notation, we define
\begin{equation*}
	F_d(q) \coloneqq \prod_{n\ge1} A_d\left(q^n\right).
\end{equation*}
We begin with a preparatory lemma on the asymptotic of $F_d(q)$.

\begin{lemma}\label{L:BigF}
	As $w \to 0$ in $D_\theta$, we have
	\begin{equation*}
		F_d\left(e^{-w}\right) = \frac{e^{\frac{C_{d}}{w}+\frac{d-1}{24}w}}{\sqrt{d!}}\left(1+O\left(w^{N}\right)\right)
	\end{equation*}
	for any $N\in\N$.
\end{lemma}
\begin{proof}
	Let
	\begin{equation*}
		\CF_{d}(q):= \Log(F_d(q)) = \sum_{n\ge1} \Log\left(A_d\left(q^n\right)\right),
	\end{equation*}
	where throughout we use the principal branch of the logarithm. Then
	\begin{equation*}
		\CF_d\left(e^{-w}\right) = \sum_{n\ge1} f_d(nw),
	\end{equation*}
	where
	\begin{equation*}
		f_{d}(z) := \Log\left(A_d\left(e^{-z}\right)\right).
	\end{equation*}
	Note that by \eqref{eqn: A_d(1)} we have $A_d(1) = d! >0$, that by \eqref{Eulerian Polynomial Definition} we have $A_d(0) = 1$, and that $A_d(e^{-w})$ is holomorphic in $w$. Therefore, in the limit $w \to 0$ (i.e., for $|w|$ suitably small) we have that $A_d(e^{-nw})$ is arbitrarily close to $d!$, and avoids the branch of the complex logarithm on the cut $(-\infty,0]$.
	
	Recall that $A_d(x)$ has no roots on the interval $[0,1]$ by Lemma \ref{L:A_d Zeros}. Applying Proposition \ref{Theorem:EulerMaclaurin1DShifted} gives that for $w\to 0$ in $D_\theta$ we have that
	\begin{equation}\label{eqn: CF asymp}
		\CF_d\left(e^{-w}\right) \sim \frac{I_{f_d}}w - \sum_{n\ge0} \frac{B_{n+1}(1)}{(n+1)!} f_d^{(n)}(0) w^n.
	\end{equation}
	We adopt the usual convention that $f(z) \sim \sum_{n \geq -1} a_n z^n$ means that for each $N \geq -1$, we have $f(z) = \sum_{n = -1}^N $ $ a_n z^n + O(z^{N+1})$.
	
	To determine the term $n=0$, we compute, using \eqref{eqn: A_d(1)},
	\begin{equation*}
		f_d(0) = \log(A_d(1)) = \log(d!).
	\end{equation*}
	Thus the term $n=0$ in \eqref{eqn: CF asymp} equals $-\frac{\log(d!)}2$.
	We next determine the term $n=1$. By definition
	\begin{equation*}
		f_d'(0) = \left[\frac\del{\del z} \log\left(A_d\left(e^{-z}\right)\right)\right]_{z=0} = -\frac{A_d'(1)}{A_d(1)}.
	\end{equation*}
	Using \eqref{eqn: A_d(1)} and \eqref{Eulerian Polynomial Special Values}, we obtain $\frac{d-1}{24}$ for the term $n=1$.
	Plugging into \eqref{eqn: CF asymp} we therefore obtain
	\begin{align*}
		\mathcal F_d\lp e^{-w} \rp \sim \dfrac{I_{f_d}}{w} - \dfrac{\log(d!)}{2} + \dfrac{d-1}{24}w + O\left(w^{2}\right).
	\end{align*}

	We are left to show that the asymptotic expansion has no further terms than the three given on the right-hand side. Using \eqref{Eulerian Polynomial Symmetry}, it is not difficult to show that
	\begin{equation*}
		f_{d}^{\ast}(z) := f_{d}(z) + \frac{d-1}{2}z = \log\left(A_{d}\left(e^{-z}\right)\right) +\frac{d-1}{2}z
	\end{equation*}
	is an even function. Therefore, in \eqref{eqn: CF asymp} only the term $n=1$ and $n$ even terms survive. However, for $n \geq 2$ even it is well-known that $B_{n+1}(1)=0$, and thus only the terms $n=0$ and $n=1$ in the sum of \eqref{eqn: CF asymp} contribute to the asymptotic. Combining these observations gives the claim.
\end{proof}

We are now in a position to prove Theorem \ref{Thm: s_d}.
\begin{proof}[Proof of Theorem \ref{Thm: s_d}]
	Recall the generating function in \eqref{eqn: gen function s_d}. Using \eqref{E:partas} and Lemma \ref{L:BigF}, we have as $w \to 0$ in $D_\theta$ that
	\begin{equation*}
		\sum_{n\ge0} s_d(n) e^{-nw} \sim \left(\sqrt{\frac w{2\pi}} e^{\frac{\pi^2}{6w}}\right)^{d+1} \frac{e^{\frac{C_d}w+\frac{d-1}{24}w}}{\sqrt{d!}} \sim \frac1{(2\pi)^\frac{d+1}2\sqrt{d!}} w^{\frac{d+1}2} e^{\left(C_d+\frac{\pi^2(d+1)}6\right)\frac1w}.
	\end{equation*}
	Plugging into Proposition \ref{P:CorToIngham}, with $\l=\frac1{(2\pi)^\frac{d+1}2\sqrt{d!}}$, $\b=\frac{d+1}2$, and $\g=C_d+\frac{\pi^2(d+1)}6$ then gives the claimed asymptotic for $s_d(n)$.
\end{proof}

Using similar techniques, it is not hard to prove the following theorem for a general class of polynomials. Note that in general we do not obtain a terminating asymptotic expansion.

\begin{theorem}\label{Thm: general polynomials}
	 Let $P(x) \in \R[x]$ be a monic polynomial with $P(0)=1$, $P(1)>0$, and assume that $P$ has no zeros in $[0,1]$. Let $H(q) \coloneqq \prod_{n\ge1} P(q^n)$. Denote the Fourier coefficients of $H(q)$ by $c(n)$. Suppose that $c(n)$ are non-negative and weakly increasing for $n \gg 0$. Define
	\begin{align*}
		\mathcal{C}_P \coloneqq \int_0^\infty \log\lp P(e^{-x}) \rp dx.
	\end{align*}
	Then as $n\to \infty$ we have that
	\begin{align*}
		c(n) \sim \frac{\mathcal{C}_P^{\frac14}}{2\sqrt{\pi P(1)} n^{\frac34}} e^{2\sqrt{\mathcal{C}_P n}}.
	\end{align*}
\end{theorem}

\begin{remarks*}\hspace{0cm}
	\begin{enumerate}[wide,labelindent=0pt,labelwidth=!,label=\rm(\arabic*)]
		\item One may use Proposition \ref{Main Term Integrals} to obtain that 
		\begin{align*}
		\mathcal{C}_P = - \sum_{j=1}^d \mathrm{Li}_2\lp \frac{1}{\alpha_j} \rp,
		\end{align*}
		where the sum runs over all roots $\alpha_j$ of $P$ counted with multiplicity.
		\item The results can be extended immediately to products of rational functions, provided the numerator and denominator satisfy the hypotheses. This is done in more generality in Theorem \ref{Thm: General Theorem}. One could also avoid the need for monotonicity of the coefficients if stronger asymptotic properties away from $q \to 1$ are derived.
		\item For certain choices of polynomial $P$, it is not hard to see that the sum of dilogarithms defining $\mathcal{C}_P$ simplifies considerably. For example, let $\ell$ be a fixed prime. If $P$ is chosen to be the $\ell$-th cyclotomic polynomial $\Phi_\ell$, the roots are precisely all of the primitive $\ell$-th roots of unity. Then using the distribution property for dilogarithms (see e.g. \cite[page 9]{Z}), we recover the $\ell$-regular partition asymptotic. This agrees with the asymptotic arising from the $\ell$-regular partition generating function
		\begin{align*}
		\prod_{n\geq 1} \Phi_\ell(q^n) = \prod_{n\geq 1} \frac{1-q^{\ell n}}{1-q^n}
		\end{align*}
		where the asymptotic for the coefficients of the right-hand side can be evaluated using standard techniques - see e.g. \cite{chern}. We discuss the possibility of finding simpler expressions for $\mathcal C_P$ in more generality in Section \ref{Final Remarks}.
	\end{enumerate}
\end{remarks*}


\section{Proof of Theorem \ref{Thm: r_d}} \label{Proof of Theorem 2}

In this section we prove Theorem \ref{Thm: r_d}. Recall the generating function for $r_d(n)$ given in \eqref{eqn: gen function r_d}. To ease notation, we let
\begin{equation*}
	G_{d}(q) \coloneqq \prod_{n\ge0}F_{d}\left(q^{(d+1)n+1},q\right),
\end{equation*}
where $F_{d}(x,y)$ is defined in Subsection 2.3.

We again being with a preparatory lemma on the asymptotic of $G_d(q)$.
\begin{lemma}\label{Lem: G_d}
	As $w \to 0$ in $D_{\theta}$, we have that
	\begin{align*}
		G_d\left( e^{-w}\right) \sim \frac{e^{	\frac{C_d}{(d+1)w} + \left(\frac{1}{2} - \frac{1}{d+1}\right) d!}}{\left(d!\right)^{\frac{d}{2(d+1)}}} .
	\end{align*}
\end{lemma}

\begin{proof}
	We have
	\begin{equation*}
		\CG_{d}(q):=\Log(G_{d}(q)) = \sum_{n\ge0}\Log\left(F_{d}\left(q^{(d+1)n+1},q \right)\right).
	\end{equation*}
	Write
	\begin{equation*}
		\CG_{d}\left(e^{-w}\right) = \sum_{n\ge0}g_{d,w}\left(\left(n + \frac{1}{d+1}\right)(d+1)w\right),
	\end{equation*}
	where 
	\begin{equation*}
		g_{d,w}(x):=\Log\left(F_{d}\left(e^{-z},e^{-w}\right)\right).
	\end{equation*}
		
	Observe that $F_d(0,0) = 1$ and that by Lemma \ref{L:F_d zeros}, $F_d\lp e^{-tw}, e^{-w} \rp$ does not vanish for $|w|$ small and $t \in \R_0^+$. Note that by Lemma \ref{F_d(x,1) = A_d(x)} and \eqref{eqn: A_d(1)} we have that $F_d(1,1) = A_d(1) = d! >0 $ and that $F_d(e^{-tw},e^{-w})$ is holomorphic in $w$ for $t\in\R_0^+$. Therefore, in the limit $w \to 0$ (i.e., for $|w|$ suitably small) we again avoid the branch of the complex logarithm on the cut $(-\infty,0]$.
	
	Then applying Proposition \ref{Thm: Euler-Mac full} with $N=2$ gives that
	\begin{align}\label{eqn: asymp of CG}
		\mathcal G_d\left(e^{-w}\right) &= \dfrac{I_{g_{d,w}}}{(d+1)w} - \sum_{n=0}^1 \dfrac{B_{n+1}\lp \frac{1}{d+1} \rp g_{d,w}^{(n)}\lp 0 \rp}{(n+1)!} (d+1)^n w^n \nonumber\\
		&\hspace{0.2cm}- \frac1{d+1} \sum_{k \geq 2} \dfrac{g_{d,w}^{(k)}(0)}{(k+1)!} w^k - \dfrac{w^2}{2\pi i} \sum_{n=0}^1 \dfrac{(d+1)^n B_{n+1}(0)}{(n+1)!} \int_{C_R(0)} \dfrac{g_{d,w}^{(n)}(z)}{z^{2-n} (z-w)} dz \nonumber\\
		&\hspace{0.2cm}- \frac{(d+1)w}{2} \int_w^{w\infty} g_{d,w}^{\prime\prime}(z) \widetilde B_2\lp \frac{z}{(d+1)w} - \frac{1}{d+1} \rp dz.
	\end{align}
	The main asymptotic contribution comes from the term
	\begin{equation*}
		\frac{I_{g_{d,0}}}{(d+1)w} = \frac{C_d}{(d+1)w}
	\end{equation*}
	by Lemma \ref{F_d(x,1) = A_d(x)}.
	
	The constant term in \eqref{eqn: asymp of CG} is 
	\begin{equation*}
		\frac{I_{\left[\frac\partial{\partial w}g_{d,w}\right]_{w=0}}}{d+1} - B_{1}\left(\frac{1}{d+1}\right)g_{d,0}(0).
	\end{equation*}
	We have by Lemma \ref{F_d(x,1) = A_d(x)} and \eqref{eqn: A_d(1)} that
	\begin{equation*}
		g_{d,0}(0) = A_{d}(1) = d!.
	\end{equation*}
	We then compute
	\begin{equation*}
		\left[\frac{\partial}{\partial w} g_{d,w}(x)\right]_{w=0} = \left[\frac{\partial}{\partial w} \log\left(F_{d}\left(e^{-x},e^{-w}\right)\right)\right]_{w=0} = - \frac{F_{d}^{(0,1)}\left(e^{-x},1\right)}{F_{d}\left(e^{-x},1\right)}.
	\end{equation*}
	Using Lemma \ref{F_d Differential Equation} we therefore obtain that
	\begin{multline*}
		I_{\left[\frac\partial{\partial w}g_{d,w}\right]_{w=0}} dx = -\int_0^\infty \frac{F_d^{(0,1)}\left(e^{-x},1\right)}{F_d\left(e^{-x},1\right)} dx = \frac d2\int_0^\infty \frac{\frac\del{\del x}F_d(e^{-x},1)}{F_d\left(e^{-x},1\right)} dx\\
		= \frac d2\int_0^\infty \frac\del{\del x} \log\left(F_d\left(e^{-x},1\right)\right) dx = \frac d2(\log(F_d(0,1))-\log(F_d(1,1))).
	\end{multline*}
	We now claim that $F_d(0,1)=1$. By Lemma \ref{F_d(x,1) = A_d(x)}, we have $F_d(0,1)=A_d(0)$. We plug into \eqref{Eulerian Polynomial Recursion} and obtain
	\begin{equation*}
		A_d(0) = A_{d-1}(0) = 1,
	\end{equation*}
	as $A_1(0)=1$. Moreover, again by Lemma \ref{F_d(x,1) = A_d(x)} and \eqref{eqn: A_d(1)}, we have
	\begin{equation*}
		F_d(1,1) = A_d(1) = d!.
	\end{equation*}
	Thus,
	\begin{equation*}
		I_{\left[\frac\partial{\partial w}g_{d,w}\right]_{w=0}} = -\frac d2\log(d!).
	\end{equation*}
	So the constant term in \eqref{eqn: asymp of CG} is equal to
	\begin{align*}
		-\frac{d\log(d!)}{2(d+1)} + \left(\frac{1}{2} -\frac{1}{d+1}\right) d!.
	\end{align*}
	Exponentiating gives the claim, noting that the remaining terms go into the error.
\end{proof}

We are now in a position to prove Theorem \ref{Thm: r_d}.

\begin{proof}[Proof of Theorem \ref{Thm: r_d}]
	Recall the generating function for $r_{d}(n)$ in \eqref{eqn: gen function r_d}.
	We then use Lemma \ref{Lem: G_d} and \eqref{E:partas} to obtain that
	\begin{align*}
		\sum_{n\geq0} r_d(n)e^{-nw} \sim \sqrt{\frac w{2\pi}} e^{\frac{\pi^2}{6w}} \frac{e^{	\frac{C_d}{(d+1)w} + \left(\frac{1}{2} - \frac{1}{d+1}\right) d!}}{\left(d!\right)^{\frac{d}{2(d+1)}}} 
		=\frac{e^{\frac{d-1}{2(d-1)}d!}}{\sqrt{2\pi}(d!)^{\frac{d}{2(d+1)}}}\sqrt{w}e^{\left(\frac{C_{d}}{d+1}+\frac{\pi^{2}}{6}\right)\frac{1}{w}}
	\end{align*}
as $w \to 0$ in $D_\theta$. Applying Proposition \ref{P:CorToIngham} with $\lambda = (2\pi)^{-\frac{1}{2}} e^{\frac{d-1}{2(d+1)}d!} d!^{-\frac{d}{2(d+1)}}$, $\beta = \frac{1}{2}$, and $\gamma = \frac{C_d}{d+1} + \frac{\pi^2}{6}$ gives the claim.
\end{proof}

We note here that asymptotic ``tricks'' used to prove Theorem \ref{Thm: r_d} can be generalized quite broadly. The main point is that Lemma \ref{Lem: G_d} can be greatly generalized to many products of the form $\prod_{n \geq 0} P\lp q^n, q \rp$ where $P(x,y) \in \R[x,y]$. The basic idea is to take a logarithm in order to reduce the question to asymptotics for $\Log\lp P\lp e^{-nz}, e^{-z} \rp \rp$. The main idea of our method is to use the Euler--Maclaurin formula as stated in Proposition \ref{Thm: Euler-Mac full} to give an exact formula for $\Log\lp P\lp e^{-z}, e^{-w} \rp \rp$ for $w$ fixed. Then, suitable holomorphic properties of this expression permit the substitution $w = z$ as $z \to 0$, and we can then compute suitable asymptotics as $z \to 0$. We execute these objectives in the following two results.

\begin{lemma} \label{L: Asymp}
	Let $P \in \R[x,y]$ be a polynomial such that $P(1,1) > 0$, $P(0,1) = 1$, and such that $P(x,1)$ has no zeros for $0 \leq x \leq 1$. Let $a$, $b\in\N$ with $0\leq a<b$ and
	\begin{align*}
		G_P(q) := \prod_{n \geq 0} P\lp q^{bn+a}, q \rp.
	\end{align*}
	Then in each region $D_\theta$ with $0 < \theta < \frac{\pi}{2}$, we have as $w \to 0$ in $D_\theta$ that
	\begin{align*}
		G_P\lp e^{-w} \rp \sim P(1,1)^{\frac 12 - \frac ab} e^{\mathcal D_P} \cdot e^{\frac{\mathcal C_{\mathcal P}}{bw}}
	\end{align*}
	where we define $\mathcal{P}(x) := P(x,1)$, $\mathcal C_P$ as in Theorem \ref{Thm: general polynomials} and
	\begin{align*}
		\mathcal D_P:= I_{\left[\frac\partial{\partial w}g_{P,w}\right]_{w=0}} = - \int_0^\infty \dfrac{P^{(0,1)}\lp e^{-x}, 1 \rp}{P\lp e^{-x}, 1 \rp} dx.
	\end{align*}
\end{lemma}

\begin{proof}
	We consider $\mathcal G_P(q) := \Log\lp G_P(q) \rp$. Then we have
	\begin{align*}
		\mathcal G_P\lp e^{-w} \rp = \sum_{n \geq 0} g_{P,w}\lp \lp n + \frac{a}{b} \rp bw \rp,
	\end{align*}
	where for any fixed $w$ we define
	\begin{align*}
		g_{P,w}\lp z \rp := \Log\lp P\lp e^{-z}, e^{-w} \rp \rp.
	\end{align*}
	As in previous results, the conditions we assume for $P(x,y)$ ensure that $g_{P,w}(z)$ satisfies the analytic conditions necessary for convergence of $\mathcal G_P\lp e^{-w} \rp$ and the application of Proposition \ref{Thm: Euler-Mac full}. By applying Proposition \ref{Thm: Euler-Mac full} in this setting, we see that
	\begin{align*}\label{eqn: asymp G_P}
		\mathcal G_P\left(e^{-w}\right) &= \dfrac{I_{g_{P,w}}}{bw} - \sum_{n=0}^1 \dfrac{B_{n+1}\lp \frac{a}{b} \rp g_{P,w}^{(n)}\lp 0 \rp}{(n+1)!} (bw)^n \nonumber\\
		&\hspace{0.2cm}- \frac{1}{b} \sum_{k \geq 2} \dfrac{g_{P,w}^{(k)}(0) a^{k+1}}{(k+1)!} w^k - \dfrac{(bw)^2}{2\pi i} \sum_{n=0}^1 \dfrac{\lp \frac ab \rp^{2-n} B_{n+1}(0)}{(n+1)!} \int_{C_R(0)} \dfrac{g_{P,w}^{(n)}(z)}{z^{2-n} (z-aw)} dz \nonumber\\
		&\hspace{0.2cm}+ \frac{bw}{2} \int_{aw}^{w\infty} g_{P,w}^{\prime\prime}(z) \widetilde B_2\lp \frac{z}{bw} - \dfrac{a}{b} \rp dz.
	\end{align*}
	In order to apply Proposition \ref{P:CorToIngham}, we need the terms up through the constant term in $w$ of this expansion; we see that
	\begin{align*}
		\mathcal G_P\lp e^{-w} \rp &= \dfrac{I_{g_{P,w}}}{bw} - B_1\lp \frac ab \rp g_{P,w}(0) + O\lp w \rp \\ &= \dfrac{\mathcal C_{\mathcal P}}{bw} + \mathcal D_P + \lp \dfrac{1}{2} - \dfrac{a}{b} \rp \Log\lp P(1,1) \rp + O(w).
	\end{align*}
	This completes the proof.
\end{proof}

On the basis of this lemma, we can prove asymptotic formulas for the coefficients of these very general rational products.

\begin{theorem} \label{Thm: General Theorem}
	Let $P,Q \in \R[x,y]$ be polynomials such that $P(1,1), Q(1,1) > 0$, $P(0,1) = Q(0,1) = 1$, and such that $P(x,1)$ and $Q(x,1)$ have no zeros for $0 \leq x \leq 1$. Let
	\begin{align*}
	H(q) := \sum_{n \geq 0} c(n) q^n := \prod_{n \geq 0} \dfrac{P\lp q^{A n + a}, q \rp}{Q\lp q^{B n + b}, q \rp},
	\end{align*}
	and suppose that for $n \gg 0$ the $c(n)$ are increasing functions. Then as long as $\frac{\mathcal C_{\mathcal{P}}}{A} > \frac{\mathcal C_{\mathcal{Q}}}{B}$, we have, as $n \to \infty$,
	\begin{align*}
		c(n) \sim \frac{\l_{P,Q,a,A,b,B} \cdot \mathcal C_{P,Q,A,B}^{\frac{1}{4}}}{2\sqrt\pi n^{\frac{3}{4}}}e^{2\sqrt{\mathcal C_{P,Q,A,B} n}},
	\end{align*}
	where we define
	\begin{align*}
		\mathcal C_{\mathcal{P},\mathcal{Q},A,B} := \dfrac{\mathcal C_{\mathcal{P}}}{A} - \dfrac{\mathcal C_{\mathcal{Q}}}{B} >0, \ \ \ \mathcal D_{P,Q} := \mathcal D_P - \mathcal D_Q, \ \\
		\lambda_{P,Q,a,A,b,B} := P(1,1)^{\frac 12 - \frac aA} Q(1,1)^{\frac bB - \frac 12} e^{\mathcal D_{P,Q}}.
	\end{align*}
\end{theorem}

\begin{proof}
	Because we assume that $c(n)$ is increasing for $n \gg 0$, we need only calculate suitable asymptotics as $z \to 0$ in regions $D_\theta$ for certain $0 < \theta < \frac{\pi}{2}$. From Lemma \ref{L: Asymp} we obtain
	\begin{align*}
		\prod_{n \geq 0} \dfrac{P\lp q^{A n + a}, q \rp}{Q\lp q^{B n + b}, q \rp} \sim  P(1,1)^{\frac 12 - \frac aA} Q(1,1)^{\frac bB - \frac 12} \cdot e^{\lp \frac{\mathcal C_P}{A} - \frac{\mathcal C_Q}{B} \rp \frac{1}{w} + \mathcal D_{P,Q}}.
	\end{align*}
	This completes the proof by Proposition \ref{P:CorToIngham} with $\lambda = \lambda_{P,Q,a,A,b,B}$, $\beta = 0$, and $\gamma = \mathcal C_{P,Q,A,B} >0$.
\end{proof}

\section{Final Remarks} \label{Final Remarks}

\subsection{Applications in mathematical physics}
There are potential applications of our method to the computation of asymptotic formulas for coefficients of thermal partition functions in super Yang--Mills theory \cite{AMMPR,BCDM,M}. For example the partition functions in \cite[equation (5.10)]{AMMPR}, \cite[equation (2.6)]{BCDM}, and \cite[equation (7.3)]{M}, can all be treated with this approach. These partition functions often take the form of infinite products over rational functions evaluated at $q^n$. Asymptotic formulas for these partition functions can be used to derive information about the entropy of the relevant system. Although the physics literature does contain some elementary methods for computing asymptotics for these coefficients, our method is capable of vast generalization. In particular, since our asymptotic method is based upon the exact formula given in Proposition \ref{Thm: Euler-Mac full}, one could compute asymptotics to much higher degrees of precision, and therefore obtain more accurate entropies.

\subsection{Dilogarithms and Bloch groups}

Let $P$ be a polynomial of degree $d$ with integral coefficients with $P(0)=1$ and no roots in the interval $[0,1]$. If $\alpha_1, \alpha_2, \dots, \alpha_d \in \C\backslash[0,1]$ are the zeros of $P$, then by Proposition \ref{Main Term Integrals} we have
\begin{align*}
	\mathcal C_P = - \sum_{j=1}^d \mathrm{Li}_2\lp \dfrac{1}{\alpha_j} \rp.
\end{align*}
It is natural to ask whether there is a simpler representation for $\mathcal C_P$. Since the values $\frac1{\alpha_j}$ are the zeros of the reciprocal polynomial $z^d P(\frac1z)$, we reframe this question in the following slightly more general way.

\begin{question*}
	Let $P(z) \in \Z[z]$ be a monic polynomial. Then under what circumstances does the value
	\begin{align*}
		\sum_{P(\alpha) = 0} \mathrm{Li}_2\lp \alpha \rp
	\end{align*}
	simplify in some sense?
\end{question*}

It has been pointed out to the authors by Zagier that this somewhat vague question is closely connected to the so-called Bloch group. This is defined in terms of the {\it Bloch--Wigner dilogarithm function} 
$$
	D(z) := \mathrm{Im}\lp \mathrm{Li}_2(z) \rp + \mathrm{Arg}\lp 1-z \rp \log |z|,
$$
which is real-analytic for $z \in \C\backslash \{ 0,1 \}$ and satisfies 
$$
	D(z) = -D(1-z) = -D\lp \dfrac 1z \rp
$$ 
along with a five-term relation \cite[p. 11]{Z}
\begin{align*}
	D\lp x \rp + D\lp y \rp + D\lp \dfrac{1-x}{1-xy} \rp + D\lp 1-xy \rp + D\lp \dfrac{1-y}{1-xy} \rp = 0.
\end{align*}
Motivated by connections to the volumes of hyperbolic 3-manifolds (as explained in \cite{Z}), Bloch defined a group structure based on this functional equation. To be more precise, the {\it Bloch group} of a field $L \subseteq \overline{\Q}$, denoted $\mathcal B_L$, is defined as all formal linear combinations of symbols $\left[ \alpha \right], \alpha \in L^\times \setminus\{1\}$ subject to the relations
\begin{align*}
	\left[ x \right] + \left[ \dfrac{1}{x} \right] = \left[ x \right] + \left[ 1 - x \right] = \left[ x \right] + \left[ y \right] + \left[ \dfrac{1-x}{1-xy} \right] + \left[ 1 - xy \right] + \left[ \dfrac{1-y}{1-xy} \right] = 0.
\end{align*}
Now, $\mathcal B_L$ is certainly abelian and countable, its rank is the number of pairs of complex embeddings of $L$ into $\C$, and nontrivial elements are easy to produce \cite[p. 15--16]{Z}. Torsion elements have been well-studied, and the basic result is that for $\xi \in L$, we have $\left[ \xi \right] \in \mathcal B_L$ is torsion if and only if $D\lp \xi^\sigma \rp = 0$ for all complex embeddings $\sigma \in \mathrm{Gal}\lp L/\mathbb{Q} \rp$ (see Section B on page 36 of \cite{Z}). If we compare this fact with the properties of the so-called Rogers dilogarithm \cite[p. 23]{Z}, then we see that this is equivalent to $\mathrm{Li}_2\lp \xi \rp \in \log (\overline{\mathbb{Q}})^{\otimes 2}$, that is, that $\mathrm{Li}_2\lp \xi \rp$ can be expressed as linear combinations of forms $\log\lp \alpha \rp \log\lp \beta \rp$ for $\alpha, \beta \in \overline{\mathbb{Q}}$. Therefore, the previous question about simplified values of $\mathcal C_P$ motivates the following question:
\begin{question*}
	Let $P \in \Z[x]$ be a monic polynomial with splitting field $L$. Then under what circumstances is
	\begin{align*}
		\sum_{P\lp \alpha \rp = 0} \left[ \alpha \right] \in \mathcal B_L
	\end{align*}
	a torsion element in the Bloch group $\mathcal B_L$, and if it is torsion, what is its order?
\end{question*}
For the Eulerian polynomials $A_d(x)$ in particular, by using \label{Eulerian Polynomial Symmetry} and the dilogarithm identity $\mathrm{Li}_2(x) + \mathrm{Li}_2\lp \frac{1}{x} \rp = -\zeta(2) - \frac{\log(-x)^2}{2}$, it can be shown that $\mathcal C_{A_d}$ is, up to an explicit multiple of $\zeta(2)$, an element of $\log\lp \overline{\Q} \rp^{\otimes 2}$. It would be quite interesting to understand this question more deeply using the tools of Bloch groups and the dilogarithm identities.

\end{document}